\documentclass[12pt]{article}

\usepackage{amsmath}
\usepackage{amsfonts}
\usepackage{amssymb}
\usepackage{amsthm}
\usepackage{breqn}
\usepackage{setspace}
\usepackage{fullpage}
\usepackage{enumitem}
\usepackage{bbold} 
\usepackage{comment}
\usepackage{hyperref}
\usepackage[T1]{fontenc}

\newtheorem{lemma}{Lemma}
\newtheorem{theorem}{Theorem} 
\newtheorem{corollary}{Corollary} 
\newtheorem{definition}{Definition}

\newtheorem{observation}{Observation}

\newcommand{\F}{\mathcal{F}}

\newcommand{\h}{\mathcal{H}}

\newcommand{\abs}[1]{\left\lvert{#1}\right\rvert}
\newcommand{\floor}[1]{\left\lfloor{#1}\right\rfloor}
\newcommand{\ceil}[1]{\left\lceil{#1}\right\rceil}

\title{On the maximum size of connected hypergraphs without a path of given length}
\author
{
Ervin Gy\H{o}ri \thanks{Alfr\'ed R\'enyi Institute of Mathematics, Hungarian Academy of Sciences.  e-mail: gyori.ervin@renyi.mta.hu}
\and
Abhishek Methuku \thanks{ Central European University, Budapest. e-mail: abhishekmethuku@gmail.com}
\and
Nika Salia \thanks{ Central European University, Budapest. e-mail: Nika\char`_Salia@phd.ceu.edu}
\and 
Casey Tompkins  \thanks{Alfr\'ed R\'enyi Institute of Mathematics, Hungarian Academy of Sciences. e-mail: ctompkins496@gmail.com}
\and
M\'at\'e Vizer \thanks{Alfr\'ed R\'enyi Institute of Mathematics, Hungarian Academy of Sciences. e-mail: vizermate@gmail.com.}
}

\begin{document}
\maketitle
\begin{abstract}
In this note we asymptotically determine the maximum number of hyperedges possible in an $r$-uniform, connected $n$-vertex hypergraph without a Berge path of length $k$, as $n$ and $k$ tend to infinity.  We show that, unlike in the graph case, the multiplicative constant is smaller with the assumption of connectivity.   
\end{abstract}

\section{Introduction}
Let $P_k$ denote a path consisting of $k$ edges in a graph $G$.  There are several notions of paths in hypergraphs the most basic of which is due to Berge.  A Berge path of length $k$ is a set of $k+1$ distinct vertices $v_1,v_2,\dots,v_{k+1}$ and $k$ distinct hyperedges $h_1, h_2, \dots, h_k$ such that for $1 \le i \le k$, $v_i,v_{i+1} \in h_i$.  A Berge path is also denoted simply as $P_k$, and the vertices $v_i$ are called basic vertices.  If $v_1=v$ and $v_{k+1}=w$, then we call the Berge path a Berge $v$-$w$-path.    A hypergraph $\h$ is called connected if for any $v \in V(\h)$ and $w\in V(\h)$ there is a Berge $v$-$w$-path. Let $N_s(G)$ denote the number of $s$-vertex cliques in the graph $G$.

\vspace{2mm}

A classical result of Erd\H{o}s and Gallai \cite{erdHos1959maximal} asserts that
\begin{theorem}[Erd\H{o}s-Gallai]
\label{EGpaths}
Let $G$ be a graph on $n$ vertices not containing $P_k$ as a subgraph, then
\begin{displaymath}
\abs{E(G)} \le \frac{(k-1)n}{2}.
\end{displaymath}
\end{theorem}
In fact, Erd\H{o}s and Gallai deduced this result as a corollary of the following stronger result about cycles,
\begin{theorem}[Erd\H{o}s-Gallai]
\label{EGcycles}
Let $G$ be a graph on $n$ vertices with no cycle of length at least $k$, then 
\begin{displaymath}
\abs{E(G)} \le \frac{(k-1)(n-1)}{2}.
\end{displaymath}
\end{theorem}
Kopylov \cite{kopylov1977maximal} and later  Balister, Gy\H{o}ri, Lehel and Schelp \cite{balister2008connected} determined the maximum number of edges possible in a connected $P_k$-free graph. 

\begin{theorem}
\label{connEG}
Let $G$ be a connected $n$-vertex graph with no $P_k$, $n>k \ge 3$. Then $\abs{E(G)}$ is bounded above by $$\max \{ \binom{k-1}{2} + n-k+1, \binom{\ceil{\frac{k+1}{2}}}{2}+\floor{\frac{k-1}{2}}(n-\ceil{\frac{k+1}{2}}) \}.$$
\end{theorem}
Observe that, although the upper bound is lower in the connected case, it is nonetheless the same asymptotically. Balister, Gy\H{o}ri, Lehel and Schelp also determined the extremal cases.   
\begin{definition}
The graph $H_{n,k,a}$ consists of 3 disjoint vertex sets $A,B,C$ with $\abs{A} = a$, $\abs{B} = n-k+a$ and $\abs{C} = k-2a$. $H_{n,k,a}$ contains all edges in $A \cup C$ and all edges between $A$ and $B$. $B$ is taken to be an independent set.  The number of $s$-cliques in this graph is
\begin{displaymath}
f_s(n,k,a) = \binom{k-a}{s} + (n-k+a)\binom{a}{s-1}.
\end{displaymath}
\end{definition}
The upper bound of Theorem \ref{connEG} is attained for the graph $H_{n,k,1}$ or $H_{n,k,\floor{\frac{k-1}{2}}}$.

We now mention some recent results of Luo \cite{luo2017maximum} which will be essential in our proof.  
\begin{theorem}[Luo]
\label{luo1}
Let $n-1 \ge k \ge 4$.  Let $G$ be a connected $n$-vertex graph with no $P_k$, then the number of $s$-cliques in $G$ is at most
\begin{displaymath}
\max\{f_s(n,k,\floor{(k-1)/2}), f_s(n,k,1)\}.
\end{displaymath}
\end{theorem}
\noindent As a corollary, she also showed
\begin{corollary}[Luo]
\label{luo2}
Let $n \ge k \ge 3$.  Assume that $G$ is an $n$-vertex graph with no cycle of length $k$ or more, then
\begin{displaymath}
N_s(G) \le \frac{n-1}{k-2}\binom{k-1}{s}.
\end{displaymath}
\end{corollary}
\noindent

\noindent
Gy\H{o}ri, Katona and Lemons \cite{gyHori2016hypergraph} initiated the study of Berge $P_k$-free hypergraphs.  They proved
\begin{theorem}[Gy\H{o}ri-Katona-Lemons]
\label{HGEG}
Let $\h$ be an $r$-uniform hypergraph with no Berge path of length $k$.  If $k>r+1>3$, we have
\begin{displaymath}
\abs{E(\h)} \le \frac{n}{k} \binom{k}{r}.
\end{displaymath}
If $r \ge k>2$, we have
\begin{displaymath}
\abs{E(\h)} \le \frac{n(k-1)}{r+1}.
\end{displaymath}
\end{theorem}

\noindent
The case when $k=r+1$ was settled later \cite{davoodi2016erd}:
\begin{theorem}[Davoodi-Gy\H{o}ri-Methuku-Tompkins]
\label{main}
Let $\h$ be an $n$-vertex $r$-uniform hypergraph.   If $\abs{E(\h)} > n$, then $\h$ contains a Berge path of length at least $r+1$.
\end{theorem}

Our main result is the asymptotic upper bound for the connected version of Theorem \ref{HGEG}, as $n$ and $k$ tend to infinity.  
\begin{theorem}
\label{main}
Let $\h_{n,k}$ be a largest  $r$-uniform connected $n$-vertex hypergraph with no Berge path of length $k$, then
\begin{displaymath}
\lim_{k \to \infty} \lim_{n \to \infty}  \frac{\abs{E(\h_{n,k})}}{k^{r-1}n} = \frac{1}{2^{r-1}(r-1)!}.
\end{displaymath} 
\end{theorem}
A construction yielding the bound in Theorem \ref{main} is given by partitioning an $n$-vertex set into two classes $A$, of size $\floor{ \frac{k-1}{2}}$, and $B$, of size $n -\floor{ \frac{k-1}{2}}$ and taking $X \cup \{y\}$ as a hyperedge for every $(r-1)$-element subset $X$ of $A$ and every element $y \in B$. This hypergraph has no Berge $P_k$ as we could have at most $\floor{ \frac{k-1}{2}}$ basic vertices in $A$ and $\floor{ \frac{k-1}{2}} +1$ basic vertices in $B$, thus yielding less than the required $k+1$ basic vertices.  

Observe that in Theorem \ref{HGEG} the corresponding limiting value of the constant factor is $\frac{1}{r!}$ which is $\frac{2^{r-1}}{r}$ times larger than in the connected case. 
Note that the ideas of the proof of Theorem \ref{main} can be used to prove that the limiting value of the constant factor in Theorem \ref{HGEG} is $\frac{1}{r!}$.
\section{Proof of Theorem \ref{main}}

We will use the following simple corollary of Theorem \ref{luo1}. 
\begin{corollary}
\label{Luocor}
Let $G$ be a connected graph on $n$ vertices with no $P_k$, then $G$ has at most $$\frac{k^{r-1} n}{2^{r-1}(r-1)!}$$ $r$-cliques if $n \ge c_{k,r}$ for some constant $c_{k,r}$ depending only on $k$ and $r$.
\end{corollary}
\begin{proof}
From Theorem \ref{luo1}, it follows that for large enough $n$, the number of $r$-cliques is at most  
\begin{displaymath}
\left(n - \floor{\frac{k-1}{2}}\right)\binom{\floor{\frac{k-1}{2}}}{r-1} + \binom{\floor{\frac{k-1}{2}}}{r} + \binom{\floor{\frac{k-1}{2}}}{r-2} < n\binom{\frac{k}{2}}{r-1}.\qedhere\end{displaymath} 
\end{proof}

Given an $r$-uniform hypergraph $\h$ we define the shadow graph of $\h$, denoted $\partial\h$ to be the graph on the same vertex set with edge set:
\begin{displaymath}
E(\partial\h) := \{\{x,y\}:\{x,y\} \subset e \in E(\h)\}.
\end{displaymath}

\begin{definition}
If $r=3$, then we call an edge $e \in E(\partial\h)$ \emph{fat} if there are at least 2 distinct hyperedges $h_1,h_2$ with $e \subset h_1,h_2$.  If $r>3$, then we call an edge $e \in E(\partial\h)$ \emph{fat} if there are at least $k$ distinct hyperedges $h_1,h_2,\dots,h_k$ in $\h$ with $e \subset h_i$ for $1 \le i \le k$. 

\vspace{1mm}

\noindent We call an edge $e \in E(\partial\h)$ \emph{thin} if it is not fat.
\end{definition}
  Thus, the set $E(\partial\h)$ decomposes into the set of fat edges and the set of thin edges.   We will refer to the graph whose edges consist of all fat edges in $\partial\h$ as the \emph{fat graph} and denote it by $F$.   

\begin{lemma}\label{fatlem}
There is no $P_k$ in the fat graph $F$ of the hypergraph $\h$.  
\end{lemma}
\begin{proof}
Suppose we have such a $P_k$ with edges $e_1,e_2,\dots,e_k$.  For $r=3$, if a hyperedge contains two edges from the path, then it must contain consecutive edges $e_i,e_{i+1}$. Select hyperedges $h_1,h_2,\dots,h_k$ where $e_i \subset h_i$ in such a way that $h_{i+1}$ is different from $h_i$ for all $1 \le i \le k-1$, and these edges yield the required Berge path.

Suppose now that $r>3$, we will find a Berge path of length $k$ in $\h$, greedily.  For $e_1$, select an arbitrary hyperedge $h_1$ containing it.  Suppose we have found a distinct hyperedge $h_i$ containing the fat edge $e_i$ for all $1 \le i < i^*$.  Since the edge $e_{i^*}$ is fat, there are at least $k$ different hyperedges $h_{i^*}^1,h_{i^*}^2,\dots,h_{i^*}^k$ containing it.  Select one of them, say $h_{i^*}^j$, which is not equal to any of $h_1,h_2,\dots,h_{i^*-1}$.  Thus, we may find distinct hyperedges $h_1,h_2,\dots,h_k$ where $e_i \subset h_i$ for $1 \le i \le k$, and thus, we have a Berge path of length $k$.
\end{proof}
We call a hyperedge $h \in E(\h)$ \emph{fat} if $h$ contains no thin edge. Let $\F$ denote the hypergraph on the same set of vertices as $\h$ consisting of the fat hyperedges, then

\begin{lemma}
\label{HminusF}
If $r=3$, then
\begin{displaymath}
\abs{E(\h \setminus \F)} \le \frac{(k-1)n}{2}.
\end{displaymath}
If $r>3$, then
\begin{displaymath}
\abs{E(\h \setminus \F)} \le \frac{(k-1)^2n}{2}.
\end{displaymath}
\end{lemma}
\begin{proof}
Arbitrarily select a thin edge from each $h \in \h \setminus \F$.  Let $G$ be the graph consisting of the selected thin edges.  We know that each edge in $G$ was selected at most once if $r=3$ and at most $k-1$ times in the $r>3$.  Thus, we have that $\abs{\h \setminus \F} \le \abs{E(G)}$ for $r=3$ and $\abs{\h \setminus \F} \le (k-1)\abs{E(G)}$ for $r>3$.  Moreover, $G$ is $P_k$-free since a $P_k$ in $G$ would imply a Berge $P_k$ in $\h$ by considering any hyperedge from which each edge was selected.  It follows by Theorem \ref{EGpaths} that $\abs{E(G)} \le \frac{(k-1)n}{2}$, so $\abs{\h \setminus \F} \le \frac{(k-1)n}{2}$ if $r=3$, and  $\abs{\h \setminus \F} \le \frac{(k-1)^2n}{2}$ if $r>3$.
\end{proof}

Any hyperedge of $\F$ contains only fat edges, so it corresponds to a unique $r$-clique in $F$. This implies the following.

\begin{observation}
\label{cliques_in_fat_graph}
The number of hyperedges in $E(\F)$ is at most the number of $r$-cliques in the fat graph $F$.  
\end{observation}

To this end we will upper bound the number of $r$-cliques in $F$, by making use of the following important lemma.
\begin{lemma}
\label{no_disjoint_cycles}
There are no two disjoint cycles of length at least $k/2+1$ in the fat graph $F$.  
\end{lemma}
\begin{proof}
Let $C$ and $D$ be two such cycles.  By connectivity, there are vertices $v \in V(C)$ and $w \in V(D)$ and a Berge path from $v$ to $w$ in $\h$ containing no additional vertices of $C$ or $D$ as defining vertices.  This path can be extended using the hyperedges containing the edges of $C$ and $D$ to produce a Berge path of length $k$ in $\h$ (note that here we used that the edges of $C$ and $D$ are fat), a contradiction.
\end{proof}
Assume that $F$ has connected components $C_1,C_2,\dots,C_t$.  Trivially,
\begin{equation}
\label{sum_components}
N_r(F) = \sum_{i=1}^t N_r(C_i).
\end{equation}

\noindent
If $\abs{V(C_i)} \le k/2$, then trivially $$N_r(C_i) \le \binom{\abs{V(C_i)}}{r} \le \frac{\abs{V(C_i)}^r}{r!} \le \frac{k^{r-1} \abs{V(C_i)}}{2^{r-1}(r-1)!}.$$ So we can assume $\abs{V(C_i)} \ge k/2$. By Lemma \ref{no_disjoint_cycles}, we have that for all but at most one $i$, $C_i$ does not contain a cycle of length at least $k/2+1$. So by  Corollary \ref{luo2}, for all but at most one $i$, say $i_0$, we have $$N_r(C_i) \le \frac{\abs{V(C_i)}-1}{k/2-2}\binom{k/2-1}{r} \le \frac{k^{r-1} \abs{V(C_i)}}{2^{r-1}(r-1)!} + O(k^{r-2}).$$ If $\abs{V(C_{i_0})} \ge c_{k,r}$, then by Lemma \ref{fatlem} and by Corollary \ref{Luocor} we have $$N_r(C_{i_0}) \le \frac{k^{r-1} \abs{V(C_i)}}{2^{r-1}(r-1)!}.$$ Otherwise, $N_r(C_{i_0}) \le \binom{\abs{V(C_{i_0})}}{r} = o(n)$. Therefore, by \eqref{sum_components}, we have $$N_r(F) = \sum_{i=1}^t N_r(C_i) \le$$ $$ \le \sum_{i=1}^t\left(\frac{k^{r-1} \abs{V(C_i)}}{2^{r-1}(r-1)!} + O(k^{r-2})\right) + o(n) \le \frac{k^{r-1} n}{2^{r-1}(r-1)!} + O(k^{r-2})n + o(n).$$

Therefore, by Observation \ref{cliques_in_fat_graph}, 
\begin{equation}
\label{boundonscriptF}
\abs{E(\F)} \le N_r(F) \le \frac{k^{r-1} n}{2^{r-1}(r-1)!} + O(k^{r-2})n + o(n).
\end{equation}

Since $\abs{E(\h)} = \abs{E(\h \setminus \F)} + \abs{E(\F)}$, adding up the upper bounds in \eqref{boundonscriptF} and Lemma \ref{HminusF}, we obtain 
the desired upper bound on $\abs{E(\h)}$.

\qed


\section*{Acknowledgements}

The authors Gy\H{o}ri, Methuku, Salia and Tompkins were supported by the National Research, Development and Innovation Office -- NKFIH under the grant K116769.

\noindent
The author Vizer was supported by the Hungarian National Research, Development and Innovation Office -- NKFIH under the grant SNN 116095.

\bibliography{OrderMag.bib}

\end{document}